\documentclass[11pt]{amsart}

\newtheorem{theorem}{Theorem}[section]
\newtheorem{lemma}[theorem]{Lemma}
\newtheorem{corollary}[theorem]{Corollary}
\newtheorem{proposition}[theorem]{Proposition}

\theoremstyle{definition}
\newtheorem{definition}[theorem]{Definition}

\theoremstyle{remark}
\newtheorem{remark}[theorem]{Remark}

\numberwithin{equation}{section}
\usepackage{prettyref}
\usepackage{latexsym}
\usepackage{amsmath}
\usepackage{amssymb}
\usepackage{color}
\usepackage[
colorlinks=true,
linkcolor=blue,
 citecolor=blue,
  urlcolor=blue,
]{hyperref}

\begin{document}
\def\Xint#1{\mathchoice
  {\XXint\displaystyle\textstyle{#1}}%
  {\XXint\textstyle\scriptstyle{#1}}%
  {\XXint\scriptstyle\scriptscriptstyle{#1}}%
  {\XXint\scriptscriptstyle\scriptscriptstyle{#1}}%
  \!\int}
\def\XXint#1#2#3{{\setbox0=\hbox{$#1{#2#3}{\int}$}
    \vcenter{\hbox{$#2#3$}}\kern-.5\wd0}}
\def\ddashint{\Xint=}
\def\avgint{\Xint-}

\title[Quantitative two weight theorem]{  A new quantitative two weight theorem for the Hardy-Littlewood maximal operator} 

\author{Carlos P\'erez}
\address{Departamento de An\'alisis Matem\'atico,
Facultad de Matem\'aticas, Universidad de Sevilla, 41080 Sevilla,
Spain} 
\email{carlosperez@us.es}

\author{Ezequiel Rela}
\address{Departamento de An\'alisis Matem\'atico,
Facultad de Matem\'aticas, Universidad de Sevilla, 41080 Sevilla,
Spain} 
\email{erela@us.es}

\thanks{Both authors are supported by the Spanish Ministry of Science and Innovation grant MTM2012-30748 and by the Junta de Andaluc\'ia, grant FQM-4745.}



\begin{abstract} 
A quantitative two weight theorem for the Hardy-Little\-wood maximal operator is proved improving the known ones. As a consequence a new proof of the main results in \cite{HP} and \cite{HPR1} is obtained which avoids the use of the sharp quantitative reverse Holder inequality  for $A_{\infty}$ proved in those papers. Our results are valid within the context of spaces of homogeneous type without imposing the non-empty annuli condition.
\end{abstract}

\maketitle

\section{ Introduction and Main results} \label{sec:intro}

\subsection{Introduction}
The purpose of this note is to present a \emph{quantitative} two weight theorem for the Hardy-Littlewood maximal operator  when the underlying space is a space of homogeneous type $\mathcal{S}$ (SHT in the sequel), endowed with a quasimetric $\rho$ and a doubling measure $\mu$ (see Section \ref{sec:SHT} for the precise definitions). 
We briefly recall some background on this problem in the \emph{euclidean} or \emph{classical} setting,  when we are working in $\mathbb{R}^n$ and we consider Lebesgue measure and euclidean metric. We also assume that in this classical setting all the  maximal operators involved and $A_p$ classes of weights are defined over cubes. Let $M$ stand for the usual uncentered Hardy-Littlewood maximal operator:
\begin{equation*}
Mf(x) = \sup_{Q\ni x } \frac{1}{|Q|}\int_{Q} |f|\,dx.
\end{equation*}
The problem of characterizing the pair of weights for which the maximal operator is bounded  between weighted Lebesgue spaces was solved by Sawyer \cite{sawyer82b}: 
To be more precise, if $1<p<\infty$ we define for any pair of weights $w,\sigma$, the (two weight) norm,
\begin{equation}  \label{MainQuestion}
\|M(\cdot \sigma)\|_{L^p(w)}:= \sup_{f\in L^p(\sigma)}  \frac{ \|M(f \sigma) \|_{L^p(w)}}{ \|f\|_{L^p(\sigma)} }
\end{equation} 
then Sawyer showed that  $\|M(\cdot \sigma)\|_{L^p(w)}$ is finite if and and only if
\begin{equation*}
 \sup_Q \frac{\int_Q (M(\chi_Q \sigma )^p \ wdx}{ \sigma(Q)}<\infty,
\end{equation*} 
where the supremum is taken over all the cubes in $\mathbb{R}^n$. 
A quantitative precise version of this result is the following: if we define 
\begin{equation*}
 [w,\sigma]_{S_p}:= \left(\frac1{\sigma(Q)} \,\int_Q M(\sigma\chi_Q)^pw\ dx\right)^{1/p}.
\end{equation*}
then
\begin{equation}\label{eq:moen}
 \|M(\cdot \sigma)\|_{L^p(w)} \sim p'[w,\sigma]_{S_p},
\end{equation} 
where $\frac{1}{p}+\frac{1}{p'}=1$. This result is due to K. Moen and can be found in \cite{Moen:Two-weight}. 

However, it is still an open problem to find a characterization more closely related to the $A_p$ condition of Muckenhoupt which is easier to use in applications.  Indeed, recall that the two weight $A_p$ condition:
\begin{equation*}
\sup_Q\left(\avgint_Q w\ dx\right)\left(\avgint_Q v^{-\frac1{p-1}}\ dx\right)^{p-1}<\infty
\end{equation*} 
is necessary for the boundedness of $M$ from $L^p(v)$ into $L^{p}(w)$ (which is clearly equivalent, setting $\sigma=v^{1-p'}$, to the two weight problem), but it is not sufficient. Therefore, the general idea is to strengthen the $A_p$ condition to make it sufficient. The first result on this direction is due to Neugebauer \cite{Neugebauer}, proving that, for any $r>1$, it is sufficient to consider the following ``power bump'' for the $A_p$ condition:
\begin{equation}\label{Neug}
\sup_Q\left(\avgint_Q w^{r}\ dx\right)^\frac{1}{r}\left(\avgint_Q v^{-\frac{r}{p-1}}\ dx\right)^{(p-1)r}<\infty
\end{equation} 

Later, the first author improved this result in \cite{perez95} by considering a different approach which allows to consider much larger classes weights. The new idea is to replace \emph{only} the average norm associated
to the weight $v^{-\frac1{p-1}}$  in \eqref{Neug} by an ``stronger'' norm which is often called a ``bump". This norm is defined
in terms of an appropriate Banach function $X$ space satisfying certain special property. This property is related to the $L^p$ boundedness of a natural maximal function related to the space. More precisely, for a given Banach function space $X$, the local $X$-average of a measurable function $f$ associated to the cube $Q$ is defined as
\begin{equation*}
 \|f\|_{X,Q}=\left\|\tau_{\ell(Q)}(f\chi_Q)\right\|_X,
\end{equation*} 
where $\tau_\delta$ is the dilation operator $\tau_\delta f(x)=f(\delta x)$, $\delta>0$ and $\ell(Q)$ stands for the sidelength of the cube $Q$. The natural maximal operator associated to the space $X$ is defined as
\begin{equation*}
M_{X}f(x)= \sup_{Q:x\in Q} \|f\|_{X,Q}
\end{equation*}
and the key property is that the maximal operator $M_{X'}$ is bounded on  $L^p(\mathbb{R}^n)$  
where $X'$ is the associate space to $X$ (see \eqref{bnessX'} below).

As a corollary of our main result, Theorem \ref{thm:main},  we will give a quantitative version of the main result from \cite{perez95} regarding sufficient conditions for the two weight inequality to hold:

\begin{theorem}\label{thm:perez-bump}
Let $w$ and $\sigma$ be a pair of weights that satisfies the condition 
\begin{equation}\label{keySufficientCondition}
\sup \left(\avgint_Q w\ dx\right) \|\sigma^{1/p'}\|^p_{X,Q} <\infty.
\end{equation}
Suppose, in addition, that the maximal operator associated to the associate  space is bounded on $L^p(\mathbb{R}^n)$:
\begin{equation}\label{bnessX'}
 M_{X'}: L^p(\mathbb{R}^n)\to L^p(\mathbb{R}^n).
\end{equation}
Then there is a finite positive constant $C$ such that:
\begin{equation*}
\|M(\cdot \sigma)\|_{L^p(w)} \leq C. 
\end{equation*} 

\end{theorem}

In this note we give a different  result of this type with the hope that it may lead to different, possible better, conditions for the two weight problem for Singular Integral Operators.

Most of the interesting examples are obtained when $X$ is an Orlicz  space $L_\Phi $ defined in term of the Young function $\Phi$ (see Section \ref{sec:SHT} for the precise definitions). In this case,  the local average with respect to $\Phi$ over a cube $Q$ is
\begin{equation*}
 \|f\|_{\Phi,Q} =\|f\|_{\Phi,Q,\mu}= \inf\left\{\lambda >0:
 \frac{1}{\mu(Q)}\int_{Q} \Phi\left(\frac{ |f|}{ \lambda } \right) 
dx \le 1\right\}
\end{equation*}
where $\mu$ is here the Lebesgue measure. The  corresponding maximal function is 
\begin{equation}\label{eq:maximaltype}
M_{\Phi}f(x)= \sup_{Q:x\in Q} \|f\|_{\Phi,Q}.
\end{equation}

Related to condition \eqref{keySufficientCondition} we introduce here the following quantities.

\begin{definition}\label{def:Ap-multiple}
Let $(\mathcal{S},d\mu)$ be a SHT. Given a ball $B\subset \mathcal{S}$, a Young function $\Phi$ and two weights $w$ and $\sigma$, we define the quantity
\begin{equation}\label{eq:A_p-local}
 A_p(w,\sigma,B,\Phi):=\left( \avgint_{B} w\, d\mu\right)\|\sigma^{1/p'}\|^p_{\Phi,B}
\end{equation} 
and we say that a pair of weights belong to the $A_{p,\Phi}$ class if 
\begin{equation*}
[w,\sigma,\Phi]_{A_p}:=\sup_B A_p(w,\sigma,B,\Phi) <\infty,
\end{equation*}
where the $\sup$ is taken over all balls in the space. In the particular case of $\Phi(t)=t^{p'}$, this condition  corresponds to the classical $A_p$ condition and we use the notation
\begin{equation*}
 [w,\sigma]_{A_p}:=\sup_B\left(\avgint_{B} w\ d\mu\right)\left(\avgint_{B} \sigma\ d\mu\right)^{p-1}. 
\end{equation*}

\end{definition}

We define now a generalization of the Fuji-Wilson constant of a $A_{\infty}$ weight $\sigma$ as introduced  in \cite{HP}  by means of a Young function $\Phi$:
\begin{equation*}
[\sigma,\Phi]_{W_p}:=\sup_B\frac{1}{\sigma(B)}\int_B M_{\Phi}\left(\sigma^{1/p}\chi_B\right)^p\ d\mu 
\end{equation*} 

Note that the particular choice of $\Phi_p(t):=t^p$ reduces to the $A_\infty$ constant ( see Definition \eqref{eq:Ainfty} from Section \ref{sec:SHT}):
\begin{equation}\label{eq:WpPhi-p--Ainfty}
[\sigma,\Phi_p]_{W_p}=\sup_B\frac{1}{\sigma(B)}\int_B M\left(\sigma\chi_B\right)\ d\mu =[\sigma]_{A_\infty}.
\end{equation}

\subsection{Main results}
Our main purpose in the present note is to address the problem mentioned above within the context of spaces of homogeneous type. In this context, the Hardy--Littlewood maximal operator $M$ is defined over balls:
\begin{equation}\label{eq:maximal-SHT}
Mf(x) = \sup_{B\ni x } \frac{1}{\mu(B)}\int_{B} |f|\,d\mu.
\end{equation}

The Orlicz type maximal operators are defined also with balls and with respect to the measure $\mu$ in the natural way. 

Our main result is the following theorem.
\begin{theorem} \label{thm:main}
Let  $1 < p < \infty$ and let $\Phi$ be any Young function with conjugate function $\bar\Phi$. Then, for any pair of weights $w,\sigma$,
there exists a structural constant $C>0$ such that the (two weight) norm defined in \eqref{MainQuestion} satisfies
\begin{equation}\label{eq:main}
\|M(\cdot \sigma)\|_{L^p(w)} \leq C p'\left( [w,\sigma,\Phi]_{A_p}[\sigma,\bar\Phi]_{W_p}\right)^{1/p},
\end{equation}
\end{theorem}

We emphasize that \eqref{eq:main}, which is even new in the usual context of Euclidean Spaces, fits into the spirit of the $A_p-A_{\infty}$ theorem derived in \cite{HP} and \cite{HPR1}. The main point here is that we have a two weight result with a better condition and with a  proof that avoids completely the use of the sharp quantitative reverse H\"older inequality  for $A_{\infty}$ weights proved in these papers. This property is, of course, of independent interest but it is not used in our results. 

From this Theorem, we derive several corollaries. First, we have a direct proof of the two weight result derived in \cite{HP} using the $[w]_{A_{\infty}}$ constant of Fujii-Wilson \eqref{eq:Ainfty}. 

\begin{corollary}\label{cor:mixed-two-weight}
Under the same hypothesis of Theorem \ref{thm:main}, we have that there exists a structural constant $C>0$ such that
\begin{equation}\label{eq:mixed-two}
\|M(\cdot \sigma)\|_{L^p(w)} \leq Cp'\left([w,\sigma]_{A_p}[\sigma]_{A_\infty}\right)^{1/p}.
\end{equation} 
\end{corollary}

Note that the result in Theorem \ref{thm:main} involves two suprema like in Corollary \ref{cor:mixed-two-weight}.  It would interesting to find out if there is a version of this result involving only one supremum. There is some evidence that it could be the case, see for example \cite{HP}, Theorem 4.3. See also the recent work \cite{LM}.

As a second consequence of  Theorem \ref{thm:main}, we have the announced quantitative version of Theorem \ref{thm:perez-bump}:

\begin{corollary}\label{cor:precise-bump}
Under the same hypothesis of Theorem \ref{thm:main}, we have that there exists a structural constant $C>0$ such that
\begin{equation*}
\|M(\cdot \sigma)\|_{L^p(w)} \leq C p' [w,\sigma,\Phi]_{A_p}^{1/p} \|M_{\bar\Phi}\|_{L^p(\mathbb{R}^n)}
\end{equation*}
\end{corollary}
We remark that this approach produces a non-optimal dependence on $p$, since we have to pay with one $p'$ for using Sawyer's theorem. However, the ideas from the proof of Theorem \ref{thm:main} can be used to derive a direct proof of Corollary \ref{cor:precise-bump} without the $p'$ factor. We include the proof in the appendix.

Finally, for the one weight problem, we recover the known mixed bound.
\begin{corollary}\label{cor:mixed-one-weight}
For any $A_p$ weight $w$ the following mixed bound holds:
\begin{equation*}
\|M\|_{L^p(w)} \leq  C p' \left([w]_{A_p}[\sigma]_{A_\infty}\right)^{1/p}
\end{equation*} 
where $C$ is an structural constant and as usual $\sigma=w^{1-p'}$ is the dual weight.
\end{corollary}

\begin{remark} 
To be able to extend the proofs to this general scenario, we need to use (and prove) suitable versions of classical tools on this subject, such as Calder\'on--Zygmund decompositions. We remark that in previous works (\cite{PW-JFA}, \cite{SW}) most of the results are proved under the assumption that the space has non-empty annuli. The main consequence of this property is that in that case the measure $\mu$ enjoys a reverse doubling property, which is crucial in the proof of Calder\'on--Zygmund type lemmas. However, this assumption implies, for instance, that the space has infinite measure and no atoms (i.e. points with positive measure) and therefore constraints the family of spaces under study.  Recently, some of those results were proven without this hypothesis, see for example \cite{Pradolini-Salinas}. Here we choose to work without the annuli property and therefore we need to adapt the proofs from \cite{PW-JFA}. Hence, we will need to consider separately the cases when the space has finite or infinite measure. An important and useful result on this matter is the following:
\end{remark}

\begin{lemma}[\cite{vili}]\label{lem:bounded-finite}
Let  $(\mathcal S,\rho,\mu)$ be a space of homogeneous type. Then $\mathcal{S}$ is  bounded if and only if $\mu(\mathcal S)<\infty$. 
\end{lemma}

\subsection{Outline}
The article is organized as follows. In Section \ref{sec:prelim} we summarize some basic needed results on spaces of homogeneous type and Orlicz spaces. We also include a Calder\'on--Zygmund type decomposition lemma. In Section \ref{sec:proofs} we present the proofs of our results. Finally, we include in Section  \ref{sec:appendix} an Appendix with a direct proof of a slightly better result than Corollary \ref{cor:mixed-two-weight}.

\section{preliminaries}\label{sec:prelim}

In this section we first summarize some basic aspects regarding spaces of homogeneous type and Orlicz spaces. Then, we include a Calder\'on--Zygmund (C--Z) decomposition lemma adapted to our purposes.

\subsection{Spaces of homogeneous type}\label{sec:SHT}
A quasimetric $d$ on a set $\mathcal{S}$ is a function $d:{\mathcal S} \times
{\mathcal S} \rightarrow [0,\infty)$ which satisfies
\begin{enumerate}
 \item $d(x,y)=0$ if and only if $x=y$;
\item $d(x,y)=d(y,x)$ for all $x,y$;
 \item there exists a finite constant $\kappa \ge 1$ such that, for all $x,y,z \in \mathcal{S}$,
\begin{equation*}
d(x,y)\le \kappa (d(x,z)+d(z,y)).
\end{equation*}
\end{enumerate}

Given $x \in \mathcal{S}$ and $r > 0$,  we define the ball with center $x$ and radius $r$, $B(x,r) := \{y \in {\mathcal{S}} :d(x,y) < r\}$ and we denote its radius $r$ by $r(B)$ and its center $x$ by $x_B$. 
A space of homogeneous type $({\mathcal{S}},d,\mu)$ is a set $\mathcal{S}$ endowed with a quasimetric $d$ and a doubling  nonnegative Borel measure $\mu$ such that 
\begin{equation}\label{eq:doubling}
 \mu(B(x,2r)) \le C\mu(B(x,r))
\end{equation}

Let $C_\mu$ be the smallest constant satisfying \eqref{eq:doubling}. Then $D_\mu = \log_2 C_\mu$ is called the doubling order of $\mu$. It follows that 
\begin{equation}
\frac{\mu(B)}{\mu(\tilde{B})} \le
C^{2+\log_2\kappa}_{\mu}\left(\frac{r(B)}{r(\tilde{B})}\right)^{D_\mu} \;\mbox{for all
balls}\; \tilde{B} \subset B. 
\end{equation}
In particular for $\lambda>1$ and $B$ a  ball, we have that
\begin{equation}\label{eq:doublingDIL}
 \mu(\lambda B) \le (2\lambda)^{D_\mu}   \mu(B).
\end{equation}
Here, as usual, $\lambda B$ stands for  the dilation of a ball $B(x,\lambda r)$ with $\lambda>0$.
Throughout this paper, we will say that a constant $c=c(\kappa,\mu)>0$ is a \emph{structural constant} if it  depends only on the quasimetric constant $\kappa$ and the doubling constant $C_\mu$. 

An elementary but important property of the quasimetric is the following. Suppose that we have two balls $B_1=B(x_1,r_1)$ and $B_2=B(x_2,r_2)$ with non empty intersection. Then,
\begin{equation}\label{eq:engulfing}
 r_1\le r_2 \Longrightarrow B_1\subset \kappa(2\kappa+1)B_2.
\end{equation}
This is usually known as the ``engulfing'' property and follows directly from the quasitriangular property of the quasimetric.

In a general space of homogeneous type, the balls $ B(x,r)$ are not necessarily open, but by a theorem of
Macias and Segovia \cite{MS}, there is a continuous quasimetric
$d'$ which is equivalent to $d$ (i.e., there are positive
constants $c_{1}$ and $c_{2}$ such that $c_{1}d'(x,y)\le d(x,y)
\le c_{2}d'(x,y)$ for all $x,y \in \mathcal{S}$) for which every ball is
open. We always assume that the quasimetric $d$ is continuous and
that balls are open.

We will adopt the usual notation: if $\nu$ is a measure and $E$ is a measurable set, $\nu(E)$ denotes the $\nu$-measure of $E$. Also, if $f$ is a measurable function on $(\mathcal S,d,\mu)$ and $E$ is a measurable set, we will use the notation $f(E):=\int_E f(x)\ d\mu$. 
We also will denote the $\mu$-average of $f$ over a ball $B$ as $f_{B} = \avgint_B f d\mu$.
We recall that a weight $w$ (any non negative measurable function) satisfies the $A_p$ condition for $1<p<\infty$ if
\begin{equation*}
   [w]_{A_p}:=\sup_B\left(\avgint_B w\ d\mu\right)\left(\avgint_B w^{-\frac{1}{p-1}}\ d\mu\right)^{p-1},
\end{equation*}
where the supremum is taken over all the balls in $\mathcal{S}$. The $A_{\infty}$ class is defined in the natural way by $A_{\infty}:=\bigcup_{p>1}A_p$

This class of weights can also be characterized by means of an appropriate constant. In fact, there are various different definitions of this constant, all of them equivalent in the sense that  they define the same class of weights. Perhaps the more classical and known definition is the following  due  to Hru\v{s}\v{c}ev
\cite{Hruscev} (see also \cite{GCRdF}):
\begin{equation*}
[w]^{exp}_{A_\infty}:=\sup_B \left(\avgint_{B} w\,d\mu\right) \exp \left(\avgint_{B} \log w^{-1}\,d\mu  \right).
\end{equation*}
However, in \cite{HP} the authors use a ``new'' $A_\infty$ constant (which was originally introduced implicitly  by Fujii in \cite{Fujii} and later by Wilson in \cite{Wilson:87}), which seems to be better suited. For any $w\in A_\infty$, we define
 \begin{equation}\label{eq:Ainfty}
     [w]_{A_\infty}:= [w]^{W}_{A_\infty}:=\sup_B\frac{1}{w(B)}\int_B M(w\chi_B )\ d\mu,
 \end{equation}
where $M$ is the usual Hardy--Littlewood maximal operator. When the underlying space is $\mathbb{R}^d$, it is easy to see that $[w]_{A_\infty}\le c [w]^{exp}_{A_\infty}$ for some structural $c>0$.  In fact, it is shown in \cite{HP} that there are examples  showing that $[w]_{A_\infty}$ is much smaller than $[w]^{exp}_{A_\infty}$
The same line of ideas yields the inequality in this wider scenario. See the recent work of Beznosova and Reznikov \cite{BR} for a comprehensive and thorough study of these different $A_\infty$ constants.
We also refer the reader to the forthcoming work of Duoandikoetxea, Martin-Reyes and Ombrosi  \cite{DMRO} for a discussion regarding different definitions of $A_\infty$ classes.

\subsection{Orlicz spaces}\label{sec:Orlicz}
We recall here some basic definitions and facts about Orlicz spaces.

A function $\Phi:[0,\infty) \rightarrow [0,\infty)$ is called a Young function if it is continuous, convex, increasing and satisfies  $\Phi(0)=0$ and $\Phi(t) \rightarrow \infty$ as $t \rightarrow \infty$.  For Orlicz spaces, we are usually only concerned about the behaviour of Young functions for $t$ large. 
The space $L_{\Phi}$ is a Banach function space with the Luxemburg norm
\[
\|f\|_{\Phi} =\|f\|_{\Phi,\mu} =\inf\left\{\lambda >0: \int_{\mathcal{S}}
\Phi( \frac{ |f|}{\lambda }) \, d\mu \le 1 \right\}.
\]
Each Young function $\Phi$ has an associated complementary Young function $\bar{\Phi}$ satisfying
\begin{equation*}
t\le \Phi^{-1}(t)\bar{\Phi}^{-1}(t) \le 2t \label{propiedad}
\end{equation*}
for all $t>0$. The function $\bar{\Phi}$ is called the conjugate
of $\Phi$, and the  space $L_{\bar{\Phi}}$ is called the conjugate
space of $L_{\Phi}$. For example, if $\Phi(t) = t^p$ for $1 < p <
\infty$, then $\bar{\Phi}(t) = t^{p'}, p' = p/(p-1)$, and the
conjugate space of $L^p(\mu)$ is $L^{p'}(\mu)$.

A very important property of Orlicz spaces is the generalized
H\"older inequality
\begin{equation}\label{eq:HOLDERglobal}
\int_{\mathcal{S}} |fg|\, d\mu \le 2 \|f\|_{\Phi}\|g\|_{\bar{\Phi}}. 
\end{equation}
Now we introduce local versions of Luxemburg norms. If $\Phi$ is a Young function, let
\begin{equation*}
\|f\|_{\Phi,B} =\|f\|_{\Phi,B,\mu}= \inf\left\{\lambda >0:
\frac{1}{\mu(B)}\int_{B} \Phi\left(\frac{ |f|}{ \lambda }\right)  \,
d\mu \le 1\right\}.
\end{equation*} 
Furthermore, the local version of the generalized H\"older inequality (\ref{eq:HOLDERglobal}) is
\begin{equation}\label{eq:HOLDERlocal}
\frac{1}{\mu(B)}\int_{B}fg\, d\mu \le 2 \|f\|_{\Phi,B}\|g\|_{\bar{\Phi},B}. 
\end{equation}
Recall the definition of the maximal type operators $M_\Phi$ from \eqref{eq:maximaltype}:
\begin{equation}\label{eq:maximaltype-SHT}
M_{\Phi}f(x)= \sup_{B:x\in B} \|f\|_{\Phi,B}.
\end{equation}
An important fact related to this sort of operator is that its boundedness is related to the so called $B_p$ condition. For any  positive function  $\Phi$ (not necessarily a Young function), we have that 
\begin{equation*}
\|M_{\Phi}\|^p_{L^{p}(\mathcal{S})}\,   \leq c_{\mu,\kappa}\, \alpha_{p}(\Phi),
\end{equation*}
where $\alpha_{p}(\Phi)$ is the following tail condition 
\begin{equation}\label{eq:Phi-p}
\alpha_{p}(\Phi)= \,\int_{1}^{\infty} \frac{\Phi(t)} { t^p }
\frac{dt}{t} < \infty.
\end{equation}
It is worth noting that in the recent article \cite{LL} the authors define the appropriate analogue of the $B_p$ condition in order to characterize the boundedness of the \emph{strong} Orlicz-type maximal function  defined over rectangles both in the linear and multilinear cases. Recent developments and improvements can also be found in \cite{Masty-Perez}, where the authors addressed the problem of studying the maximal operator between Banach function spaces.

\subsection{Calder\'on--Zygmund decomposition for spaces of homogeneous type}

\

The following lemma is a classical result in the theory, regarding a decomposition of a generic level set of the Hardy--Littlewood maximal function $M$. Some variants can be found in \cite{AimarPAMS} for $M$ and in \cite{AimarTAMS} for the centered maximal function $M^c$ . In this latter case, the proof is straightforward. We include here a detailed proof for the general case of $M$ where some extra subtleties are needed.
%
%
%

\begin{lemma}[Calder\'on--Zygmund decomposition]\label{lem:stoppingtime} 
Let $B$ be a fixed ball and let $f$ be a bounded nonnegative measurable function. Let $M$ be the usual non centered Hardy--Littlewood maximal function. Define the set $\Omega_{\lambda}$ as 
\begin{equation}\label{eq:omegalambda}
\Omega_\lambda = \{x \in  B:  Mf(x) >\lambda\},
\end{equation} 
Let $\lambda>0$ be such that $\lambda\ge \avgint_B f\ d\mu$. If $\Omega_{\lambda}$ is non-empty, then given $\eta > 1$, there exists a countable family
$\{B_i\}$ of pairwise disjoint balls such that, for $\theta=4\kappa^2+\kappa$,
\begin{itemize}
\item[i)] $\displaystyle \cup_{i} B_{i}\subset \Omega_{\lambda} \subset
\cup_{i} \theta B_{i}$,  
\item[ii)] For all $i$, 
\begin{equation*}
\lambda <\frac{1}{\mu(B_i)} \int_{B_i} f d\mu.
\end{equation*}
\item[iii)] If $B$ is any ball such that $B_i\subset B$ for some $i$ and $r(B)\ge \eta r(B_i)$, we have that
\begin{equation}\label{eq:ballmaximal1}
\frac{1}{\mu(\eta B)} \int_{\eta B} f d\mu  \le \lambda.
\end{equation} 
\end{itemize}
\end{lemma}

\begin{proof}
Define, for each $x\in\Omega_\lambda$, the following set:
 \begin{equation*}
 \mathcal{R}^\lambda_x=\left\{r>0: \avgint_{B} f\ d\mu >\lambda, x\in B=B(y,r)\right\},
 \end{equation*}
which is clearly non-empty. The key here is to prove that $\mathcal{R}^\lambda_x$ is bounded. If the whole space is bounded, there is nothing to prove. In the case of unbounded spaces, we argue as follows. 
Since the space is of infinite measure (recall Lemma \ref{lem:bounded-finite}), and clearly $S=\bigcup_{r>0} B(x,r)$,
we have that $\mu(B(x,r))$ goes to $+\infty$ when $r\to\infty$ for any $x\in \mathcal{S}$. Therefore, for $K=\kappa(2\kappa+1)$, we can choose $r_1$ such that the ball $B_1=B(x,r_1)$ satisfies the inequality
\begin{equation*}
\mu(B_1)\ge \frac{2(2K)^{D_\mu}\|f\|_{L^1}}{\lambda}
\end{equation*}
Suppose now that $\sup\mathcal{R}^\lambda_x=+\infty$. Then we can choose a ball $B_2=B(y,r_2)$ for some $y$ such that $x\in B_2$, $\avgint_{B_2}f\ d\mu>\lambda$ and $r_2>r_1$. Now, by the engulfing property \eqref{eq:engulfing} we obtain that $B_1\subset KB_2$. The doubling condition \eqref{eq:doublingDIL} yields 
\begin{equation*}
 \mu(B_1)\le \mu(KB_2)\le (2k)^{D_\mu}\mu(B_2)
\end{equation*}
Then we obtain that
\begin{equation*}
\frac{2\|f\|_{L^1}}{\lambda}\le \mu(B_2)
<  \frac{\|f\|_{L^1}}{\lambda}
\end{equation*}
which is a contradiction. We conclude that, in any case, for any $x\in \Omega_\lambda$, we have that $\sup \mathcal{R}^\lambda_x<\infty$.

Now fix $\eta>1$. If $x \in \Omega_\lambda$, there is a ball $B_{x}$ containing $x$, whose radius $r(B_x)$ satisfies $\frac{\sup \mathcal{R}^\lambda_x}{\eta} < r(B_x)\leq \sup \mathcal{R}^\lambda_x$, and for which $\avgint_{B_x} f \ d\mu > \lambda$. Thus the ball $B_x$ satisfies ii) and iii). Also
note that $\Omega_\lambda = \bigcup_{x\in \Omega_{\lambda}} B_x$.
Picking a Vitali type subcover of $\{B_{x}\}_{x\in
\Omega_{\lambda}}$ as in \cite{SW}, Lemma 3.3, we obtain a
family of pairwise disjoint balls $\{B_{i}\} \subset
\{B_{x}\}_{x\in \Omega_{\lambda}}$ satisfying i).  Therefore
$\{B_i\}$ satisfies i), ii) and iii).
\end{proof}

We will need another important lemma, in order to handle simultaneously decompositions of level sets at different scales. 

\begin{lemma}\label{lem:disjointing}
Let $B$ be a ball and let $f$ be a bounded nonnegative measurable function. Let also $a \gg 1$ and, for each integer $k$ such that $a^k>\avgint_B f\ d\mu$, we define $\Omega_{k}$ as
\begin{equation}\label{eq:Omega-k}
\Omega _{k} = \left\{x\in B: Mf(x) >a^{k}  \right\}, 
\end{equation}
Let $\{E_i^k\}_{i,k}$ be defined by $E_i^k=B_i^k\setminus \Omega_{k+1}$, where the family of balls $\{B_i^k\}_{i,k}$ is obtained by applying Lemma \ref{lem:stoppingtime} to each $\Omega_k$.
Then, for $\theta=4\kappa^2+\kappa$ as in the previous Lemma and $\eta=\kappa^2(4\kappa+3)$, the following inequality holds:
\begin{equation}\label{eq:Bik vs Eik}
\mu(B_i^k\cap \Omega_{k+1})< \frac{(4\theta\eta)^{D_\mu}}{a}\mu(B_i^k).
\end{equation} 
Consequently, for sufficiently large $a$, we can obtain that
\begin{equation}\label{eq:Bik vs Eik one half}
\mu(B_i^k) \le 2\mu(E_i^k).
\end{equation}
\end{lemma}

\begin{proof}
To prove the claim, we apply Lemma \ref{lem:stoppingtime} with $\eta=\kappa^2(4\kappa+3)$. Then, by part i), we have that, for $\theta=4\kappa^2+\kappa$
 \begin{equation*}
\Omega_{k+1}\subset \bigcup_m\theta B^{k+1}_m
\end{equation*}
and then
\begin{equation}\label{eq:decompBik-k+1}
\mu(B_{i}^{k}  \cap \Omega_{k+1} )\le \sum_{m} \mu( B_{i}^{k} \cap
\theta B_{m}^{k+1} ).
\end{equation}
Suppose now that $B_i^k\cap \theta B_m^{k+1}\neq \emptyset$. We claim that $r(B_{m}^{k+1})\le r(B_{i}^{k})$. Suppose the contrary, namely $r(B_{m}^{k+1})> r(B_{i}^{k})$. Then, by property \eqref{eq:engulfing}, we can see that $B_{i}^{k} \subset \kappa^2(4\kappa+3) B_{m}^{k+1}=\eta B_{m}^{k+1}$. 
For $B=\eta B_{m}^{k+1}$, part iii) from Lemma \ref{lem:stoppingtime} gives us that the average satisfies 
\begin{equation}\label{eq:avg-etaBmk+1}
 \frac{1}{\mu(B)}\int_B f\ d\mu\le a^k. 
\end{equation} 
Now, by the properties of the family $\{B_m^{k+1}\}_m$ and the doubling condition of $\mu$, we have that, for $a>(2\eta)^{D_\mu}$,
\begin{equation}\label{eq:ak}
\frac{1}{\mu(\eta B_{m}^{k+1})}\int_{\eta B_{m}^{k+1}} f\ d\mu>\frac{a^{k+1}}{ (2\eta)^{D_\mu}}>a^k.
\end{equation}
This last inequality contradicts \eqref{eq:avg-etaBmk+1}. Then, whenever $B_i^k\cap \theta B_m^{k+1}\neq \emptyset$, we have that 
$r(B_{m}^{k+1})\le r(B_{i}^{k})$ and from that it follows that $ B_{m}^{k+1}\subset \eta B_{i}^{k}$. The sum \eqref{eq:decompBik-k+1} now becomes
\begin{eqnarray*}
\mu(B_{i}^{k}  \cap \Omega_{k+1} )& \le & \sum_{m:B_{m}^{k+1}\subset \eta B_{i}^{k}} \mu( B_{j}^{k} \cap \theta B_{m}^{k+1} )\\
& \le &  (2\theta)^{D_\mu} \sum_{m:B_{m}^{k+1}\subset \eta B_{i}^{k}} \mu(B_{m}^{k+1} )\\
&\le & \frac{(2\theta)^{D_\mu}}{a^{k+1}}\int_{\eta B_i^k}f\ d\mu
\end{eqnarray*}
since the sets $\{B_m^{k+1}\}_m$ are pairwise disjoint. Finally, by part iii) of Lemma \ref{lem:stoppingtime}, we obtain
\begin{equation*}
 \mu(B_{i}^{k}  \cap \Omega_{k+1})\le \frac{(4\theta\eta)^{D_\mu}}{a}\mu(B_i^k),
\end{equation*}
which is inequality \eqref{eq:Bik vs Eik}.
\end{proof}

\section{Proofs of the main results}\label{sec:proofs}

We present here the proof or our main results. Our starting point is a version of the sharp two weight inequality \eqref{eq:moen} valid for SHT from \cite{kairema:twoweight}:

\begin{theorem}[\cite{kairema:twoweight}]\label{thm:kairema}
Let $(\mathcal{S},\rho,\mu)$ a SHT. Then the H--L maximal operator $M$ defined by \eqref{eq:maximal-SHT} satisfies the bound
\begin{equation}\label{eq:kairema}
 \left\|M(f\sigma)\right\|_{L^p(w)}\le C p'[w,\sigma]_{S_p}\|f\|_{L^p(\sigma)},
\end{equation} 
where $[w,\sigma]_{S_p}$ is the Sawyer's condition with respect to balls:
\begin{equation}
[w,\sigma]_{S_p}:=\sup_B  \left(\frac1{\sigma(B)} \int_B M(\sigma\chi_B)^pw\ d\mu\right)^{1/p}.
\end{equation} 
\end{theorem}

We now present the proof of the main result.
\begin{proof}[Proof of Theorem \ref{thm:main}]
By Theorem \ref{thm:kairema}, we only need to prove that 
\begin{equation*}
 [w,\sigma]_{S_p}\le C [w,\sigma,\Phi]^{1/p}_{A_p}[\sigma,\bar\Phi]^{1/p}_{W_p}
\end{equation*}
for some constant $C$, for any Young function $\Phi$, for any $1<p<\infty$. Let $B$ be a fixed ball $B$ and consider the sets $\Omega_k$ from \eqref{eq:Omega-k} for the function $\sigma\chi_B$ for any $k\in \mathbb{Z}$. We remark here that in order to apply a C--Z decomposition of these sets, we need the level of the decomposition to be larger that the average over the ball. We proceed as follows. Take any $a>1$ and consider $k_0\in\mathbb{Z}$ such that
\begin{equation}\label{eq:small-average}
a^{k_0-1}< \avgint_B \sigma\ d\mu \le a^{k_0}.  
\end{equation}
Now, let $A$ be the set of the small values of the maximal function: 
\begin{equation*}
 A=\left\{x\in B: M(\sigma\chi_B)\le a\avgint_B \sigma\ d\mu\right\}.
\end{equation*}

For any $x\in B\setminus A$, we have that 
\begin{equation*}
 M(\sigma\chi_B)(x)> a\avgint_B \sigma\ d\mu>a^{k_0}\ge\avgint_B \sigma\ d\mu.
\end{equation*}

Therefore,
\begin{eqnarray*}
\int_B M(\sigma\chi_B)^p w \ d\mu  & =  & \int_A M(\sigma\chi_B)^p w \ d\mu  +\int_{B\setminus A} M(\sigma\chi_B)^p w \ d\mu  \\
& \le & a^p w(B) \left(\avgint_{B} \sigma\ d\mu\right)^p + \sum_{k\ge k_0} \int_{ \Omega_{k}\setminus \Omega_{k+1}} M(\sigma\chi_B)^p w\ d\mu\\
&= & I + II
\end{eqnarray*}

The first term $I$ can be bounded easily. By the general H\"older inequality \eqref{eq:HOLDERlocal}, we obtain
\begin{eqnarray*}
 I & \le & 2a^p\left(\avgint_B w\ d\mu \right)\|\sigma^{1/p'}\|^p_{\Phi,B}\|\sigma^{1/p}\|^p_{\bar\Phi,B}\ \mu(B)\\
&\le & 2[w,\sigma,\Phi]_{A_p}\int_B M_{\bar\Phi}(\sigma^{1/p}\chi_B)^p\ d\mu
\end{eqnarray*}

Now, for the second term $II$, we first note that

\begin{eqnarray*}
\int_{B\setminus A} M(\sigma\chi_B)^p w \ d\mu & = & \sum_{k\ge k_0} \int_{ \Omega_{k}\setminus \Omega_{k+1}} M(\sigma\chi_B)^p w\ d\mu\\
& \le & a^p\sum_{k\ge k_0} a^{kp} w(\Omega_{k})
\end{eqnarray*}

By the choice of $k_0$, we can apply Lemma \ref{lem:stoppingtime} to perform a C--Z decomposition at all levels $k\ge k_0$ and obtain a family of balls $\{B^k_i\}_{i,k}$ with the properties listed in that lemma. Then,

\begin{eqnarray*}
\int_{B\setminus A} M(\sigma\chi_B)^p w \ d\mu &\le & a^{p} \sum_{k,i} \left(\avgint_{B_i^k}\sigma\chi_B\ d\mu \right)^{p} w(\theta B_i^k)\\
&\le & a^{p} \sum_{k,i} \left(\frac{\mu(\theta B_i^k)}{\mu(B_i^k)}\avgint_{\theta B_i^k}\sigma^\frac{1}{p}\sigma^\frac{1}{p'}\chi_B\ d\mu \right)^{p} w(\theta B_i^k)
 \end{eqnarray*}
We now proceed as before, using  the local generalized Holder inequality \eqref{eq:HOLDERlocal} and  the doubling property \eqref{eq:doublingDIL} of the measure (twice). Then we obtain
\begin{equation*}
 \int_{B\setminus A} M(\sigma\chi_B)^p w  d\mu \le 2a^{p} (2\theta)^{(p+1)D_\mu}[w,\sigma,\Phi]_{A_p}\sum_{k,i}\left\| \sigma^\frac{1}{p}\chi_B\right\|^p_{\bar{\Phi},\theta B_i^k}\mu(B_i^k)
\end{equation*}
The key here is to use Lemma \ref{lem:disjointing} to pass from the family $\{B_i^k\}$ to the pairwise disjoint family $\{E_i^k\}$. Then, for $a\ge 2(4\theta\eta)^{D_\mu}$, we can bound the last sum as follows 
\begin{eqnarray*}
\sum_{k,i}\left\| \sigma^\frac{1}{p}\chi_B\right\|^p_{\bar{\Phi},\theta B_i^k}\mu(B_i^k)& \le &  2 \sum_{k,i}\left\| \sigma^\frac{1}{p}\chi_B\right\|^p_{\bar{\Phi},\theta B_i^k}\mu(E_i^k)\\
&\le&  2 \sum_{k,i} \int_{E_i^k}M_{\bar{\Phi}}(\sigma^\frac{1}{p}\chi_B)^p\ d\mu\\
&\le&  2  \int_B M_{\bar{\Phi}}(\sigma^\frac{1}{p}\chi_B)^p\ d\mu
\end{eqnarray*}
since the sets $ \{ E_{k,j}\}$ are pairwise disjoint. Collecting all previous estimates and dividing by $\sigma(B)$, we obtain the desired estimate
\begin{equation*}
 [w,\sigma]^p_{S_p}\le 4 a^{p} (2\theta)^{(p+1)D_\mu}[w,\sigma,\Phi]_{A_p} [\sigma,\bar\Phi]_{W_p},
\end{equation*}
and the proof of Theorem \ref{thm:main} is complete. 
\end{proof}

It remains to prove Corollary \ref{cor:mixed-two-weight}. To that end, we need to consider the special case of $\Phi(t)=t^{p'}$.

\begin{proof}[Proof of Corollary \ref{cor:mixed-two-weight}]
Considering then  $\Phi(t)=t^{p'}$, the quantity \eqref{eq:A_p-local} is 
\begin{eqnarray*}
A_p(w,\sigma,B,\Phi) & = &\left( \avgint_{B} w\, d\mu\right)\|\sigma^{1/p'}\|^p_{\Phi,B} \\
& = & \left( \avgint_{B} w(y)\, d\mu\right) \left( \avgint_{B} \sigma \, d\mu\right)^{p-1}.
\end{eqnarray*}
In addition, we have from \eqref{eq:WpPhi-p--Ainfty} that $[\sigma,\overline{\Phi_{p'}}]_{W_p}=[\sigma,\Phi_p]_{W_p}=[\sigma]_{A_\infty}$  and therefore we obtain \eqref{eq:mixed-two}.
\end{proof}

For the proof of Corollary \ref{cor:precise-bump}, we simply use the boundedness of $M_{\bar\Phi}$  on $L^p(\mu)$, 
\begin{equation*}
 [\sigma,\bar\Phi]_{W_p}:=\sup_B\frac{1}{\sigma(B)}\int_B M_{\bar\Phi}\left(\sigma^{1/p}\chi_B\right)^p\ d\mu 
 \leq \|M_{\bar\Phi}\|^p_{L^p}.
\end{equation*}

The proof of Corollary \ref{cor:mixed-one-weight} is trivial.

\section{Appendix}\label{sec:appendix}

We include here a direct proof of version of Corollary \ref{cor:precise-bump} which is better in terms of the dependence on $p$. Precisely, we have the following Proposition.

\begin{proposition}\label{pro:precise-bump-sharp-p}
 Let  $1 < p < \infty$. For any pair of weights $w,\sigma$ and any Young function $\Phi$,
there exists a structural constant $C>0$ such that
\begin{equation*}
\|M (f\sigma)\|_{L^p(w)}\leq C [w,\sigma,\Phi]^{1/p}_{A_p} \|M_{\bar\Phi}\|_{L^p}\|f\|_{L^{p}(\sigma)}    
\end{equation*}
\end{proposition}

\begin{proof}[Proof of Proposition  \ref{pro:precise-bump-sharp-p}]
By density it is enough to prove the inequality for each nonnegative bounded function with compact support  $f$. We first consider the case of unbounded $S$. In this case we have $\avgint_S f\sigma\ d\mu=0$. Therefore, instead of the sets from sets from \eqref{eq:Omega-k}, we consider 
\begin{equation*}\label{eq:Omega-k-global}
\Omega _{k} = \left\{x\in \mathcal{S}: M(f\sigma)(x) >a^{k}  \right\}, 
\end{equation*}
for any $a>1$ and any $k\in \mathbb{Z}$. Then, we can write

\begin{equation*}
\int_{\mathcal S} M(f\sigma)^p w \ d\mu = \sum_{k} \int_{ \Omega_{k}\setminus \Omega_{k+1}} M(f\sigma)^p w\ d\mu 
\end{equation*} 
Then, following the same line of ideas as in the proof of Theorem \ref{thm:main}, we obtain
\begin{equation*}
 \int_{\mathcal S} M(f\sigma)^p w  d\mu   \le 2a^{p} (2\theta)^{(p+1)D_\mu}[w,\sigma,\Phi]_{A_p} \sum_{k,i}\left\| f\sigma^\frac{1}{p}\right\|^p_{\bar{\Phi},\theta B_i^k}\mu(B_i^k)
\end{equation*}
By Lemma \ref{lem:disjointing} we can replace the family $\{B_i^k\}$ by the pairwise disjoint family $\{E_i^k\}$ to obtain the desired estimate:
\begin{equation*}
 \int_{\mathcal S} M(f\sigma)^p w \ d\mu \le 4a^{p} (2\theta)^{(p+1)D_\mu}[w,\sigma,\Phi]_{A_p}  \|M_{\bar{\Phi}}\|_{L^p}^p\int_{S}f^p\sigma\ d\mu.
\end{equation*}
In the bounded case, the whole space is a ball and we can write $\mathcal S=B(x,R)$ for any $x$ and some $R>0$. The problem here is to deal with the small values of $\lambda$, since we cannot apply Lemma \ref{lem:disjointing} for $a^k\le \avgint_S f\sigma\ d\mu$.  We then take any $a>1$ and consider $k_0\in\mathbb{Z}$ to verify \eqref{eq:small-average}:
\begin{equation*}
a^{k_0-1}< \avgint_S f\sigma\ d\mu \le a^{k_0}
\end{equation*}
and argue as in the proof of Theorem \ref{thm:main}.
\end{proof}

Now, from this last proposition, we can derive another proof of the mixed bound \eqref{eq:mixed-two} from Corollary \ref{cor:mixed-two-weight}. The disadvantage of this approach with respect to the previous one is that we need a deep property of $A_\infty$ weights: the sharp Reverse H\"older Inequality. In the whole generality of SHT, we only know a \emph{weak} version of this result from the recent paper \cite{HPR1}:
\begin{theorem}[Sharp  weak Reverse H\"older Inequality, \cite{HPR1}]\label{thm:SharpRHI}
Let $w\in A_\infty$. Define the exponent $r(w)=1+\frac{1}{\tau_{\kappa\mu}[w]_{A_{\infty}}}$, 
where $\tau_{\kappa\mu}$ is an structural constant.
Then,
\begin{equation*}
  \left(\avgint_B w^{r(w)}\ d\mu\right)^{1/r(w)}\leq 2(4\kappa)^{D_\mu}\avgint_{2\kappa B} w\ d\mu,
\end{equation*}
where $B$ is any ball in $\mathcal S$.
\end{theorem}
The other ingredient for the alternative proof of Corollary \ref{cor:mixed-two-weight} is  the known estimate for the operator norm for $M$. For any $1<q<\infty$, we have that $\|M\|^q_{L^q}\sim q'$. 

\begin{proof}[Another proof of Corollary \ref{cor:mixed-two-weight}]
Consider the particular choice of $\Phi(t)=t^{p'r}$ for $r>1$. Then  quantity \eqref{eq:A_p-local} is 

\begin{equation*}
A_p(w,\sigma,B,\Phi) =\left( \avgint_{B} w(y)\, d\mu\right) \left( \avgint_{B} \sigma^r\, d\mu\right)^{p/rp'}
\end{equation*}
If we choose $r$ from the sharp weak reverse H\"older property (Theorem \ref{thm:SharpRHI}), we obtain that
\begin{eqnarray*}
A_p(w,\sigma,B,\Phi) & = & 
\left( \avgint_{B} w\ d\mu\right)\left(2(4\kappa)^{D_\mu}\avgint_{2\kappa B} \sigma\ d\mu\right)^{p-1}\\
&\le& 2^{p-1}(4\kappa)^{pD_\mu}\left( \avgint_{2\kappa B} w\ d\mu\right)\left(\avgint_{2\kappa B} \sigma\ d\mu\right)^{p-1}\\
&\le & 2^{p-1}(4\kappa)^{pD_\mu}[w,\sigma]_{A_p} 
\end{eqnarray*}
And therefore the proof of Proposition \ref{pro:precise-bump-sharp-p} gives 
\begin{equation*}
\|M (f\sigma)\|_{L^p(w)} \leq C [w,\sigma]_{A_p}^{1/p}  \|M_{\bar{\Phi}}\|_{L^{p}(\mathcal{S},d\mu)}  \, \|f\|_{L^{p}(\sigma)}.
\end{equation*}
We conclude with the proof by computing $\|M_{\bar \Phi}\|_{L^p}$ for $\Phi(t)=t^{p'r}$. We use \eqref{eq:Phi-p}, and then we obtain that $\|M_{\bar \Phi}\|^p_{L^p}\le c r'p'$.
But, by the choice of $r$, it follows that $r'\sim [\sigma]_{A_\infty}$ and we obtain \eqref{eq:mixed-two}.
\end{proof}

\bibliographystyle{alpha}


\end{document}